\documentclass{amsart}
\usepackage{amsmath, amssymb,bm} 
\usepackage{xcolor}
\usepackage[verbose]{hyperref}
\hypersetup{colorlinks=false,allbordercolors=blue,pdfborderstyle={/S/U/W 1}}

\theoremstyle{thmstyleone}%
\newtheorem*{theorem*}{Theorem}
\newtheorem{theorem}{Theorem}[section]
\newtheorem{proposition}[theorem]{Proposition}%
\newtheorem{lemma}[theorem]{Lemma}%

\theoremstyle{thmstyletwo}%
\newtheorem{example}{Example}[section]%
\newtheorem{remark}{Remark}[section]%

\theoremstyle{thmstylethree}%
\newtheorem{definition}{Definition}[section]%

\title{On the validity of the Radon-Nikodym Theorem}

\author{Paolo Roselli}
\address{Dipartimento di Matematica,\\
Universit\`a degli Studi di Roma Tor Vergata,\\ 
Via della Ricerca Scientifica\\
Roma, 00133, Italy}
\email{roselli@mat.uniroma2.it}

\author{Michel Willem}
\address{Institut de Recherche en Math\'ematique et Physique,\\
Universit\'e Catholique de Louvain,\\
Chemin du Cyclotron, 2\\
Louvain la Neuve, 1348, Belgium}
\email{michel.willem@uclouvain.be}

\begin{document}

\maketitle


\begin{abstract}
This paper presents a new general formulation of the Radon-Nikodym theorem in the setting of abstract measure theory. 
We introduce the notion of weak localizability for a measure and show that this property is both necessary and sufficient for the validity of a Radon-Nikodym-type representation under a natural compatibility relation between measures. 
The proof relies solely on elementary tools, such as Markov's inequality and the monotone convergence theorem. In addition to establishing the main result, we provide a constructive approach to envelope functions for families of non-negative measurable functions supported on sets of finite measure.\end{abstract}



\begin{center}
    \textit{To the memory of Haïm Brezis, a master of the calculus of variations.}
\end{center}

\section{Introduction}\label{Intro}	

We use a variational approach,  a favorite tool of Haïm Brezis, in order to prove a necessary and sufficient condition for the validity of the Radon-Nikodym theorem.
The Radon-Nikodym theorem states that if a measure $\mu$ on a measurable space~$(\Omega, \bm{\Sigma})$ satisfies a certain property, then every measure $\nu$ that stands in a certain relation to $\mu$ can be expressed as an integral (with respect to $\mu$) of a measurable function $g = g_{\mu,\nu}$, called the density:
\begin{equation}\label{eq:RN}
\nu(A) = \int_A g \, d\mu
\end{equation}
In classical versions of the Radon-Nikodym theorem, the property required of $\mu$ is either finiteness or $\sigma$-finiteness, and the relation between $\mu$ and $\nu$ is absolute continuity (see~\cite{Halmos1974}). Investigating the conditions under which~(\ref{eq:RN}) holds has led to the discovery of new properties of the measure $\mu$ that allow for more general formulations of the Radon-Nikodym theorem.
In particular, the notion of \textit{localizable measure} was defined by Segal in~\cite{Segal1951} (see Section~\ref{sec: weak loc meas}). The measures considered by Segal are \textit{locally finite} (see~\cite{BrownPearcy1977}), i.e. \textit{measures with the finite subset property} (see~\cite{Rao2004}), or \textit{semi-finite measures} (see~\cite{DePauw2024}).
This means that for every $A\in\bm{\Sigma}$ such that $\mu(A)>0$, there exists $B\in\bm{\Sigma}$ such that $B\subseteq A$ and $0<\mu(B)<+\infty$. Any $\sigma$-finite measure is both localizable and locally finite. When $\mu$ is locally finite, the Radon-Nikodym theorem holds under natural assumptions if and only if $\mu$ is localizable \big(see Problem T(ii) at p.190 in~\cite{BrownPearcy1977},  Theorem 5 at p.325 in~\cite{Rao2004} or  Theorem~5.1 at p.301 in~\cite{Segal1951}\big).
In Section~\ref{sec: weak loc meas} we define the notion of \textit{weakly localizable measure} and we prove, in Section~\ref{sec: global theory}, that the Radon-Nikodym theorem holds (under natural assumptions on $\nu$) if and only if $\mu$ is weakly localizable. 
Of course, when $\mu$ is locally finite (i.e., semi-finite), localizability is equivalent to weak localizability. 
Our main motivation was the recent characterization by De Pauw in~\cite{DePauw2024} of the measures $\mu$ (not necessarily semi-finite) such that the canonical map 
$\bm{Y} : L^\infty(\Omega,\bm{\Sigma},\mu) \rightarrow \left[L^1(\Omega,\bm{\Sigma},\mu)\right]^*$
is surjective 
\big(let us recall that 
$\left\langle \bm{Y}\kern-3pt{\scriptstyle(f)},u\right\rangle=\int_\Omega f u \ d\mu$\big). 

\subsection{Brief summary of this work}
In Section~\ref{sec: weak loc meas}, after introducing some notations, we define the notion of weakly localizable measure and prove, for such measures, the existence of an ``upper envelope'' for any family of non-negative measurable functions that are strictly positive only on sets of finite measure.\\
In Section~\ref{sec: comp meas}, we introduce the compatibility relation between two measures defined on the same measurable space and show that this relation is necessary for the validity of~(\ref{eq:RN}).\\
In Section~\ref{sec: local theory}, we use the variational method \big(see~\cite{Halmos1974}, p.129\big) in order to prove a ``local'' Radon-Nikodym theorem.
In Section~\ref{sec: global theory}, we prove the main result.

\section{Weakly localizable measures}\label{sec: weak loc meas}

Let $\Omega $ be a set. 
Let $\bm{\Sigma}$ be a $\sigma$-algebra of subsets of $\Omega $. 
The couple $(\Omega,\bm{\Sigma})$ will be called a \textit{measurable space}.
A \textit{measure} on the measurable space $(\Omega,\bm{\Sigma})$ is a function 
$\mu : \bm{\Sigma}\rightarrow [0,+\infty]$ such
that
\begin{itemize}
	\item $\mu(\emptyset)=0$,
	\item $\mu$ is $\sigma$-additive 
	\big(i.e. if $M_1$, ...,$M_n$,... is a sequence of mutually disjoint sets in $\bm{\Sigma}$, then 
$\mu\left(\bigsqcup_n M_n\right) = \sum_{n=1}^\infty \mu(M_n)$\big).
\end{itemize}
The symbol $\sqcup$ may replace the usual symbol $\cup$ when the union concerns mutually disjoint sets.
The triple $(\Omega ,\bm{\Sigma}, \mu )$ will be called a \textit{measure space} or, briefly, a \textit{measure}.\\
Given a measure space~$(\Omega,\bm{\Sigma},\mu)$, we define 
\begin{itemize}
	\item $\mathcal{N}_\mu \doteq \big\{A \in \bm{\Sigma} : \mu (A)=0\big\}$
	\item $\mathcal{F}_\mu \doteq \big\{A \in \bm{\Sigma} : \mu (A) < +\infty\big\}$
\end{itemize}

\noindent
Let us recall the definitions of \textit{$\mu$-supremum} and of \textit{localizable measure} due to Segal~(see~\cite{Segal1951},~\cite{OkadaRicker2019}). Let us also introduce the notion of \textit{weakly localizable measure}.

\begin{definition}\label{def: mu-sup}
Given a measure space~$(\Omega,\bm{\Sigma},\mu)$,
the set $A \in \bm{\Sigma}$ is a \textit{$\mu$-supremum} of the family $\mathcal{E}\subseteq \bm{\Sigma}$ if
\begin{enumerate}
	\item[(a)] for every $E\in\mathcal{E}$, $E\setminus A \in \mathcal{N}_\mu$;
	\item[(b)]  if $B\in\bm{\Sigma}$ is such that for every $E\in\mathcal{E}$, $B\cap E \in \mathcal{N}_\mu$, then $B\cap A \in \mathcal{N}_\mu$.
\end{enumerate}

\noindent
The measure $(\Omega ,\bm{\Sigma}, \mu )$ is \textit{localizable} if every family $\mathcal{E}\subseteq \bm{\Sigma}$	has a $\mu$-supremum.

\noindent
The measure $(\Omega ,\bm{\Sigma}, \mu )$ is \textit{weakly localizable} if every family $\mathcal{E}\subseteq \mathcal{F}_\mu$	has $\mu$-supremum.
\end{definition}

It is clear that any localizable measure is weakly localizable

\begin{example}
We briefly recall that the cardinality $\#$ of any subset of a set $\Omega$ provides a measure $\mu$ in the following sense: $\mu(M)=\#M$, if the $M$ has a finite number of elements, $\mu(M)=+\infty$, otherwise. 
So, $\mathcal{N}_\#=\{\emptyset\}$, and the $\mu$-supremum, if it exists, is unique and coincides with the set-theoretic union of the family considered. Let us consider a measure which is weakly localizable, but not localizable. It suffices to take as $\Omega$ the set~$\mathbb{R}$ of real numbers; then, define $M\in\bm{\Sigma}$, if $M$ is countable or $\mathbb{R}\setminus M$ is countable. Finally, we define $\mu(M)=\#M$ if $M\subseteq \mathbb{Q}$ (the set of rational numbers), $\mu(M)=+\infty$ otherwise.
So, elements of $\mathcal{F}_\mu$ are just finite sets of rational numbers.
Then $(\mathbb{R},\bm{\Sigma},\mu)$ is weakly localizable because the $\mu$-supremum of any $\mathcal{E}\subseteq\mathcal{F}_\mu$ is a countable set (of rational numbers) which is in $\bm{\Sigma}$. On the contrary, if we allow $\mathcal{E}\subseteq \bm{\Sigma}$, then we can choose the family $\mathcal{E}=\big\{\{x\} :  x\in[0,+\infty)\big\}\subset \bm{\Sigma}$, and its only possible $\mu$-supremum, the interval $[0,+\infty)$, is not in $\bm{\Sigma}$; so, $(\mathbb{R},\bm{\Sigma},\mu)$ is not localizable.  
\end{example}

\begin{example}
The foregoing example, suggest also the example of a measure space which is not weakly localizable. One can  consider~$\Omega$ and $\bm{\Sigma}$ as in the foregoing example but, as measure, the counting measure $\#$.
This time the elements of $\mathcal{F}_\#$ are all finite sets of~$\mathbb{R}$. 
So, the family $\mathcal{E}=\big\{\{x\} :  x\in[0,+\infty)\big\}\subseteq \mathcal{F}_\#$; however, its only possible $\#$-supremum, the interval $[0,+\infty)$, is not in $\bm{\Sigma}$, as before. 
Nevertheless, if we choose as $\bm{\Sigma}=2^{\Omega}$, the family of all subsets of $\Omega$, then the counting measure $\#$ becomes localizable. This should warn the reader that the notion of (weak) localizable measure space depends on the $\sigma$-algebra considered.
\end{example}

\begin{definition}
Let us recall that, in a measurable space $(\Omega,\bm{\Sigma})$, a  function 
$f : \Omega \rightarrow [-\infty,\infty]$, is
said to be \textit{measurable} if 
$\{f>t\}=\{x \in \Omega : f(x) >t \} \in \bm{\Sigma}$ 
for every $t\in\mathbb{R}$.
\end{definition}

\begin{proposition}\label{prop: upper envelope}
Let $(\Omega , \bm{\Sigma}, \mu )$ be a weakly localizable measure.\\ 
Let $\mathcal{H}$ be a family of measurable functions 
$h:\Omega\rightarrow [0,+\infty]$  such that $\{h>0\}\in\mathcal{F}_\mu$.\\ 
Then, there exists a measurable function
$g_\mathcal{H}:\Omega \rightarrow [0,+\infty]$ such that
\begin{enumerate}
	\item[(a)] $\forall h\in\mathcal{H} \ \{h > g_\mathcal{H}\} \in \mathcal{N}_\mu$ 
	\item[(b)] if $F\in\bm{\Sigma}$ and 
$u:\Omega\rightarrow [0,{\infty}]$ 
is a non-negative measurable function such that 
$\forall h \in \mathcal{H}$\ \  $F\cap \{h > u\}\in \mathcal{N}_\mu$,
then   
$F\cap \{g_\mathcal{H} > u\}\in \mathcal{N}_\mu$.
\end{enumerate}
\end{proposition}

\begin{proof}[Property (a)]
For each $q \in\mathbb{Q}_+ = \mathbb{Q} \bigcap [0,{\infty})$ 
and each $h\in \mathcal{H}$, we define $H_q = \{h > q\}$. 

\noindent
For each fixed $q \in\mathbb{Q}_+$, we denote by $U_q$ a $\mu$-supremum of the family $\{H_q : h \in\mathcal{H}\}$.

\noindent
The function $g_\mathcal{H}:\Omega \rightarrow [0,{\infty}]$ is then defined as follows
\begin{center}
$g_\mathcal{H}(x)
=
\max\big\{0 ,\ \sup\{q \in \mathbb{Q}_+ : x \in U_q\}\big\}$.
\end{center}

\noindent
We have that for all $t \geq 0$ $\{g_\mathcal{H}>t\} 
= \bigcup\{ U_q : q \in \mathbb{Q}_+, q>t\}$;
so, the function $g_\mathcal{H}$ is measurable.

\noindent
Let us prove property (a). Let $h \in \mathcal{H}$.
 By definition of $g_\mathcal{H}$ we obtain that 

\begin{align*}
\noindent
\{h> g_\mathcal{H}\} 
& = 
\bigcup \big\{ \{h>q\}{\cap}\{q>g_\mathcal{H}\} : q\in\mathbb{Q}_+\big\} \subseteq 
\bigcup \{ H_q \setminus U_q : q \in \mathbb{Q}_+ \} \in\mathcal{N}_\mu.
\end{align*}
\end{proof}

\begin{proof}[Property (b)]
Let $F \in\bm{\Sigma}$ and 
$u:\Omega \rightarrow [0,{\infty}]$
be a measurable function such that 
$\forall h \in \mathcal{H}\ F \cap \{h>u\} \in\mathcal{N}_\mu$.

\noindent
We observe that 
\begin{center}
$
F\cap \{h >u\}
=
\bigcup \big\{ F \cap \{h>q\} \cap \{q>u\} : q\in\mathbb{Q}_+\big\}
=
\bigcup \big\{ F \cap \{q>u\} \cap H_q : q\in\mathbb{Q}_+\big\}.
$
\end{center}
So, it is clear that for every $q \in\mathbb{Q}_+$
and for every $h \in \mathcal{H}$
\begin{center}
$F \cap \{q>u\} \cap  H_q \in\mathcal{N}_\mu$.
\end{center}
By definition of $U_q$, as the $\mu$-supremum of the family $\{H_q : h \in\mathcal{H}\}$, we can apply property (b) of Definition~\ref{def: mu-sup} with $B=F\cap \{q>u\}$. As $B\cap \{h>q\} \in\mathcal{N}_\mu$ for all $h \in \mathcal{H}$, we obtain that 
\begin{center}
$F \cap  \{q>u\} \cap U_q \in\mathcal{N}_\mu$.
\end{center}

\noindent
We conclude that 
\begin{align*}
F \cap \{g_\mathcal{H}>u\}
& =
\bigcup  \big\{ F \cap \{q>u\}  \cap \{g_\mathcal{H}>q\} 
: q \in\mathbb{Q}_+\big\} \\
& \subseteq
\bigcup \big\{ F \cap \{q>u\} \cap U_q : q\in\mathbb{Q}_+\big\}
\in\mathcal{N}_\mu.
\end{align*}
\end{proof}

\begin{remark}
The arguments in the proof of the previous proposition are inspired by  some of McShane in~\cite{McShane1962}. 
\end{remark}

\section{Compatible measures}\label{sec: comp meas}

In this section we consider two relations between two measures $\mu$ and $\nu$ defined on a measurable space $(\Omega,\bm{\Sigma})$.
 
\begin{definition}
The pair $(\mu,\nu)$ is \textit{compatible} on $(\Omega,\bm{\Sigma})$ if for every $A\in\bm{\Sigma}$ such that $0<\nu(A)<+\infty$, there exists $B\in\bm{\Sigma}$ such that $B\subseteq A$, $0<\nu(B)$, and $0<\mu(B)<+\infty$.
\end{definition}

\begin{definition}
The pair $(\mu,\nu)$ satisfies the \textit{Radon-Nikodym property} if there exists a measurable function 
$g_{\mu,\nu}:\Omega \rightarrow [0,{\infty}]$
such that for every $A\in\mathcal{F}_\nu$
\[
\nu(A)
=
\int_{A} g_{\mu,\nu}\ d\mu\ .
\]
The function $g_{\mu,\nu}$ is called a \textit{density} of $\nu$ with respect to $\mu$.
\end{definition}

We shall prove that, if the pair $(\mu,\nu)$ satisfies the Radon-Nikodym property, then~$(\mu,\nu)$ is compatible.

\begin{remark}
Let us recall that $\nu$ is absolutely continuous with respect to $\mu$ on~$(\Omega,\bm{\Sigma})$ (briefly, $\nu$ is a.c. wrt $\mu$), if $\mathcal{N}_\mu \subseteq \mathcal{N}_\nu$. 
A classical result (see, for example,~\cite{Rao2004},~\cite{BrownPearcy1977}, or~\cite{Halmos1974}) asserts that if $\nu$ is a.c. wrt $\mu$, and $\mu$ is $\sigma$-finite, then the pair $(\mu,\nu)$ satisfies the Radon-Nikodym property. We remark that, if $(\mu,\nu)$ is compatible on~$(\Omega,\bm{\Sigma})$ and $A\in \mathcal{F}_\mu \cap \mathcal{F}_\nu$, then $\nu$ is a.c. wrt $\mu$ on~$(A,\bm{\Sigma}_A)$, where $\bm{\Sigma}_A=\{B\in\bm{\Sigma} : B\subseteq A\}$.
\end{remark}

\begin{example}
Let us consider the Lebesgue measure $\lambda$ on the measurable space 
$\big([0,1], \mathcal{B}\big)$ where~$\mathcal{B}$ is the family of Borel subsets  of $[0,1]$.
One can verify that 
\begin{itemize}
	\item $\lambda$ is a.c. wrt $\#$ (as any other measure does),
	\item the pair $(\lambda,\#)$ does not satisfy the Radon-Nikodym property,
	\item the pair $(\lambda,\#)$ is not compatible.
\end{itemize}
\end{example}

\begin{example}
Let $\Omega=\{a,b\}$, $\nu(\{a\})=\mu(\{a\})=1$, $\nu(\{b\})=+\infty$, and $\mu(\{b\})=0$. Then,
\begin{itemize}
	\item the pair $(\mu,\nu)$ is compatible, 
	\item $\nu$ is not a.c. wrt $\mu$,
	\item the pair $(\mu,\nu)$ satisfies the Radon-Nikodym property, with $g_{\nu,\mu}=1_{\{a\}}$.
\end{itemize}
\end{example}

\begin{theorem}\label{thm: RN implies compatib}
Let $\mu$ and $\nu$ be measures on the measurable space $(\Omega , \bm{\Sigma})$. If the pair~$(\mu,\nu)$ satisfies the Radon-Nikodym property,  then the pair $(\mu,\nu)$ is compatible on~$(\Omega,\bm{\Sigma})$.
\end{theorem}
\begin{proof}
Let $g=g_{\mu,\nu}$ be the density of $\nu$ with respect to $\mu$.
Let $A \in\bm{\Sigma}$ such that $0<\nu(A) <+\infty$.
Notice that we can write 
$A = \big(A \cap\{g =0\}\big) \sqcup \big(A\cap\{g > 0\}\big)$.

\noindent
By hypothesis, $\nu(A) = \int_A g d\mu$ and we can write 
$\nu(A) = \int_A g d\mu = \int_{A \cap \{g > 0 \}} g d\mu$.
Considering $A_n = A \cap \{g > 1/n\}$,
we have that 
$A_n \subseteq A_{n+1} \subseteq A$,
and
$\displaystyle \cup_n A_n = A \cap \{g > 0\}$. 
So, by Levi's Monontone Convergence Theorem, we have that 
\begin{center}
$\displaystyle 
\nu(A) = \int_A g d\mu = \int_{A \cap \{g> 0\}} g d\mu =
\lim_n \int_{A_n} g d\mu = \lim_n \nu(A_n)$.
\end{center}

\noindent
By the Markov inequality (see~\cite{Willem2022}), we have also that,\
 for every $n \in\mathbb{N}$,
\begin{center}
$\displaystyle
\nu(A_n)= \int_{A_n} g d\mu
\ge \frac{1}{n}\mu(A_n)$,
\end{center}
that is,
for every $n \in\mathbb{N}$, 
$\mu(A_n)\le n \int_{A_n} g d\mu  = n \nu(A_n) <+\infty$.
So, for each $n \in\mathbb{N}$ we have that $A_n \in \mathcal{F}_\mu$.
Moreover, if $m$ is large enough, we have that 
$\nu(A_m)= \int_{A_m} g  d\mu > 0$. 
\end{proof}

\begin{proposition}\label{prop: compatible implies}
Let $\mu$ and $\nu$ be measures on the measurable space $(\Omega , \bm{\Sigma})$. If the pair~$(\mu,\nu)$ is compatible on~$(\Omega , \bm{\Sigma})$ then, for every $A \in
\mathcal{F}_\nu$, we have that
\begin{center}
$\nu(A) = \sup\{ \nu(B) : B \subseteq A ,\ B \in \mathcal{F}_\mu\}$.
\end{center}
\end{proposition}

\begin{proof}
Let us consider $A \in \mathcal{F}_\nu$, and define 
$b(A)\doteq \sup\{ \nu(B) : B \subseteq A , \ B\in\mathcal{F}_\mu\}$.
Let us consider an extremizing sequence, i.e., a sequence of sets 
$B_n \subseteq A$, $B_n \in \mathcal{F}_\mu$
such that $\nu(B_n) \nearrow b(A)$.
Without loss of generality, we can assume that $B_n \subseteq B_{n+1}$.
Let us consider $B = \cup_n B_n$.
So, $B \subseteq A$, and $\displaystyle\lim_n \nu(B_n)=\nu(B)$.
Considering the set 
$D=A \setminus B$, we want to show that $\nu(D)=0$.
Assume, by contradiction, that $\nu(D)>0$.
Then, we would also have 
$\nu(D) \leq \nu(A) <+\infty$. 
Since the pair $(\mu,\nu)$ is compatible, there must exist a set $E$ such that $E \subseteq D \subseteq A$, with
$0 <\mu(E)<+\infty$, and $\nu(E)> 0$.
Note that both $D$ and $E$ are disjoint from each $B_n \subseteq B$. 
Thus, we also know that $B_n \sqcup E \in \mathcal{F}_\mu$. 
Therefore, $\displaystyle \lim_n \nu(B_n) < \lim_n \nu(B_n) + \nu(E)
= \lim_n \nu(B_n \sqcup E) \le \lim_n \nu(B_n)$: a contradiction.
Then, $\nu(D)=0$, and so 
$\displaystyle \nu(A)=\nu(B)= \lim_n \nu(B_n)= b(A)$.
\end{proof}

\section{Local Theory}\label{sec: local theory}

Let $(\Omega , \bm{\Sigma})$ be a measurable space.
The symbol $\mathcal{H}_{{\int} d\mu {\leq} \nu}$
will indicate the set of all measurable functions 
$h:\Omega \rightarrow [0,{\infty}]$ such that 

\begin{itemize}
	\item $\{h>0\}\in \mathcal{F}_\mu$, 
	\item for all $M \in \bm{\Sigma}$, $\int_M h d\mu \le \nu(M)$.
\end{itemize}
Note that $\mathcal{H}_{{\int} d\mu {\leq} \nu}\ne \emptyset$, 
as $0\in \mathcal{H}_{{\int} d\mu {\leq} \nu}$.

\begin{lemma}\label{lemma:3.1}
Let $h,k \in \mathcal{H}_{{\int}d\mu {\leq}\nu}$.
Then, 
$\max\{h,k\}= h\bigvee k \in\mathcal{H}_{{\int}d\mu \leq\nu}$. 
\end{lemma}

\begin{proof}
As $\left\{h\bigvee k > 0\right\} = \{h > 0\} \cup \{k > 0\}$, 
then also 
$\{h \bigvee k > 0\} \in\mathcal{F}_\mu$.

\noindent
Let us now consider $M \in\bm{\Sigma}$; if we define 
$N = M {\cap} \{h>k\}$ then, 

\noindent
$\int_M (h \bigvee k) d\mu
\kern-3pt
 =
\kern-3pt
\int_{N} (h \bigvee k) d\mu  + \int_{M\backslash N} (h \bigvee k) d\mu
\kern-3pt
 =
\kern-3pt
\int_{N} h d\mu  + \int_{M\backslash N} k d\mu
\kern-3pt
\le
\kern-3pt
\nu(N)+\nu(M\backslash N)
=
\nu(M)$
\end{proof}

\begin{lemma}\label{lemma: h_A}
For each 
$A \in\mathcal{F}_\mu \bigcap \mathcal{F}_\nu$, 
there exists $h_A \in \mathcal{H}_{{\int}d\mu {\leq}\nu}$
such that 
\begin{center}
$
\int_A h_A d\mu 
=
\sup\left\{ \int_A h d\mu  : h \in \mathcal{H}_{{\int}d\mu{\leq}\nu} \right\}
=
s(A)
$
\end{center}
\end{lemma}

\begin{proof}
Let $(h_n)$ be a sequence of function in 
$\mathcal{H}_{{\int}d\mu {\leq}\nu}$
 such that 
${\int}_A h_n \ d\mu\rightarrow s(A)$.
By the preceding Lemma, we can assume that 
$h_n {\leq} h_{n+1}$. 
Levi's Monotone Convergence Theorem implies that 
$
\displaystyle h_A = \lim_n h_n 1_A = \sup_n h_n 1_A $ satisfies
$\int_A h_A d\mu = s(A)$
and,
\noindent
for every $M \in\bm{\Sigma}$, 
$
\int_M h_A d\mu 
=
\int_{M\cap A} h_A d\mu
\le
\nu(M\cap A) \le\nu(M)
$, where $1_A$ is the indicator function of set $A$.
\end{proof}

\begin{lemma}\label{lemma: 3.3}
Let $A \in \mathcal{F}_\mu \bigcap \mathcal{F}_\nu$, and let~$h_A$ be the function of the foregoing lemma.
For every~$C\in\bm{\Sigma}$ such that $C\subseteq A$ and $\mu(C)>0$, and for every $\varepsilon > 0$, there exists 
$D=D_{C,\varepsilon} \in \bm{\Sigma}$ such that $D \subseteq C$ and
$\nu(D) < \int_D (h_A+\varepsilon) d\mu $.
\end{lemma}

\begin{proof}
It is clear that 
$h_A + \varepsilon 1_C \notin \mathcal{H}_{{\int}d\mu {\leq}\nu}$.
By definition of
$\mathcal{H}_{{\int}d\mu {\leq}\nu}$,
there exists $B \in\bm{\Sigma}$ such that
$\nu(B) < \int_B (h_A +\varepsilon 1_C) d\mu 
\le 
\int_{B \cap C} (h_A +\varepsilon )d\mu + \nu(B\backslash C)$.
We conclude that 
$\nu(B \cap C) < \int_{B \cap C} (h_A +\varepsilon) d\mu \ \square$
\end{proof}

\begin{proposition}\label{Lemma 3.4} 
Let $\mu$ and $\nu$ be measures on the measurable space $(\Omega , \bm{\Sigma})$. Assume the pair~$(\mu,\nu)$ is compatible on~$(\Omega,\bm{\Sigma})$ and $A\in \mathcal{F}_\mu \cap \mathcal{F}_\nu$, then 
the function $h_A$ of Lemma~\ref{lemma: h_A} satisfies 
\begin{center}
$\displaystyle
\nu(A)
=
\int_A h_A d\mu
=
\sup\left\{
\int_A h \ d\mu : h\in \mathcal{H}_{{\int}d\mu {\leq}\nu}
\right\}
$.
\end{center}
\end{proposition}

\begin{proof}
It suffices to prove that $\nu(A) \le \int_A h_A d\mu$.

\noindent
Let $\varepsilon ${\textgreater}0 and define 
$\beta(A) 
=
\sup\left\{\mu(B) : B\in\bm{\Sigma}, B \subseteq A, \nu(B)\le\int_B
(h_A + \varepsilon) d\mu\right\}
$.

\noindent
We construct inductively a sequence $(B_n)\subset \bm{\Sigma}$ such that 
\begin{itemize}
	\item $B_1 \subseteq B_2 \subseteq \cdots \subseteq A$,
	\item $\nu(B_n) \le \int_{B_n} (h_A + \varepsilon) d\mu$,
	\item $\mu(B_n) \rightarrow \beta(A)$.
\end{itemize}
We choose $B_1=\emptyset$. Given $B_n$, there exists $C_n\in\bm{\Sigma}$ such that 
\begin{itemize}
	\item $C_n \subseteq A\setminus B_n$,
	\item $\nu(C_n) \le \int_{C_n} (h_A + \varepsilon) d\mu$,
	\item $\beta(A)-\mu(B_n)-\frac{1}{n} \le \mu(C_n)$.
\end{itemize}
We define $B_{n+1} = B_n \sqcup C_n$ (as $B_n$ and $C_n$ are disjoint).
So that 
\begin{center}
$
\beta(A)-\frac{1}{n} \le \mu(B_n) +\mu(C_n)= \mu(B_{n+1}) \le \beta(A)
$.
\end{center}
Let us define $B = \bigcup_n B_n$.
It follows from Levi's Monotone Convergence Theorem that 
$\nu(B)\le \int_B (h_A + \varepsilon) d\mu $. 
Let $D = A \backslash B$. We want to show that $\mu(D)=0$.

\noindent
If $\mu(D) > 0$, then the preceding Lemma implies the existence of 
$C\subseteq D$ such that 
$\nu(C) < \int_C (h_A + \varepsilon) d\mu$.
It follows that 
\begin{center}
$\displaystyle
\beta(A)= \lim_n \mu (B_n) = \mu (B) < \mu(B)+ \mu(D) 
= \mu(B+D) \le \beta(A)$.
\end{center}

\noindent
This is a contradiction; so, we have that $\mu(D)=0$.
By compatibility of~$(\mu,\nu)$ which, in our hypothesis, corresponds to absolute continuity of $\nu$ wrt $\mu$, we have that $\nu(D)=0$.
We conclude that 
$\nu(A) = \nu(B)  \le \int_B (h_A + \varepsilon) d\mu 
= 
\int_A (h_A + \varepsilon ) d\mu$.

\noindent
As $\varepsilon >0$ is arbitrary, 
$\nu(A) \le \int_A h_A d\mu$.
\end{proof}

\begin{remark}
Lemma~\ref{lemma: h_A} is contained in the proof of Theorem~B at pag.128 of~\cite{Halmos1974}. We we gave a short proof for the sake of completeness. In order to prove that $\nu(A)=\int_A h_A\ d\mu$, Halmos uses the Hahn Decomposition Theorem. Various more elementary arguments were given by Schep in~\cite{Schep2003} and its addenda (posted on his web page, see \url{https://people.math.sc.edu/schep/} under ``Selected Reprints'').  
In our variational approach we consider $\nu(A)=\int_A h_A\ d\mu$ as the Euler equation satisfied by $h_A$.
\end{remark}

\section{Global Theory}\label{sec: global theory}
Our main result is the following.
\begin{theorem}\label{Thm: Main}
Let $(\Omega , \bm{\Sigma})$ be a measurable space. Let $\mu$ be a measure on $(\Omega, \bm{\Sigma})$.\\
Statements $(a)$ and $(b)$ are equivalent:
\begin{enumerate}
	\item[(a)] {\small $\mu$ is weakly localizable;}
	\item[(b)] {\small for every measure $\nu$ on~$(\Omega, \bm{\Sigma})$, if the pair $(\mu,\nu)$ is  compatible on~$(\Omega, \bm{\Sigma})$, 
	then 
	the pair $(\mu,\nu)$ satisfies the Radon-Nikodym property.}
\end{enumerate}
\end{theorem}

\begin{proof}[(a)$\Rightarrow $(b)]

\noindent
Let us assume that $\mu $ is weakly localizable, and that the pair 
$(\mu,\nu)$ is compatible.\\
Let us consider (see section~\ref{sec: local theory}) 
$\mathcal{H}_{{\int}d\mu {\leq}\nu}$ as the set of all non-negative measurable functions $h$ such that having 
$\{h>0\}\in\mathcal{F}_\mu$, 
and for all 
$M \in\bm{\Sigma}$, $\int_M h d\mu \le \nu(M)$.

\noindent
Setting $\mathcal{H}=\mathcal{H}_{{\int}d\mu {\leq}\nu}$, Proposition~\ref{prop: upper envelope} implies the existence of measurable function 
$g_\mathcal{H}:\Omega \rightarrow [0,{\infty}]$ such that\\ 
(i)\ \ $\forall h \in \mathcal{H}$ \ \  $\{h > g_\mathcal{H} \}\in \mathcal{N}_\mu$\\ 
(ii)\ \ if $F\in\bm{\Sigma}$ and $u:\Omega \rightarrow [0,+\infty]$ is a measurable function such that 
$\forall h\in \mathcal{H}$ \ 
$F \cap \{h>u\} \in \mathcal{N}_\mu$, 
then $F \cap \{g_\mathcal{H}>u\} \in \mathcal{N}_\mu$.\\
Let $A\in\mathcal{F}_\nu$. Then, we have that
\begin{align*}
\nu(A) & = \sup_{
		\begin{array}{c}
			\scriptstyle B \subseteq A\\ 
			\scriptstyle B \in \mathcal{F}_\mu
		\end{array}} \nu(B)
\hspace{20pt} \textrm{(by Proposition~\ref{prop: compatible implies})}\\
			& = \sup_{
		\begin{array}{c}
			\scriptstyle B \subseteq A\\ 
			\scriptstyle B \in \mathcal{F}_\mu
		\end{array}} \sup_{h\in\mathcal{H}} \int_B h \ d\mu
\hspace{20pt} \textrm{(by Proposition~\ref{Lemma 3.4})}\\
			& = \sup_{h\in\mathcal{H}} 
			\sup_{
		\begin{array}{c}
			\scriptstyle B \subseteq A\\ 
			\scriptstyle B \in \mathcal{F}_\mu
		\end{array}} \int_B h \ d\mu
		\le \sup_{h\in\mathcal{H}} \int_A h \ d\mu\\
		& \le \nu(A)
		\hspace{20pt} \textrm{(by definition of~$\mathcal{H}$)}
\end{align*}
Hence, there exists a sequence $(h_n)$ in $\mathcal{H}$ such that 
$\displaystyle \int_A h_n \ d\mu \rightarrow \nu(A)$.\\
By Lemma~\ref{lemma:3.1}, we can assume that $h_n\le h_{n+1}$.
Levi's Monotone Convergence Theorem implies that 
$\displaystyle h_A=\lim_n h_n 1_A = \sup_n h_n 1_A$ satisfies 
$\displaystyle \int_A h_A\ d\mu=\nu(A)$.\\
For every $n$, we have that $\{h_n> g_{\mathcal{H}}\}\in\mathcal{N}_\mu$.
So, we obtain
\begin{center}
$
\displaystyle
\{h_A> g_{\mathcal{H}}\}
=
\bigcup_n \{h_n> g_{\mathcal{H}}\}\in\mathcal{N}_\mu
$.
\end{center}
For every $h\in\mathcal{H}$, we have that
\begin{center}
$
\displaystyle
\nu(A)
=
\sup_n\int_A h_n\ d\mu
\le
\sup_n\int_A \big(h_n\vee h\big)\ d\mu
\le
\nu(A)
$
\end{center}
It follows from Levi's Theorem that
\begin{center}
$
\displaystyle
\nu(A)
=
\int_A h_A\ d\mu
\le
\int_A \sup_n\big(h_n\vee h\big)\ d\mu
\le
\nu(A)
$
\end{center}
We conclude that $A\cap \{h> h_A\} \in \mathcal{N}_\mu$.\\
Since $h\in\mathcal{H}$ is arbitrary, we have that 
$A\cap \{g_\mathcal{H}> h_A\} \in \mathcal{N}_\mu$.
We conclude that 
\begin{center}
$
\displaystyle
\nu(A)
=
\int_A h_A \ d\mu
=
\int_A g_\mathcal{H}\ d\mu
$.
\end{center}

As $A \in \mathcal{F}_\nu$ is arbitrary, then the pair $(\mu,\nu)$ satisfies the Radon-Nikodym property.
\end{proof}

\begin{remark}
Notice that, while in Proposition~\ref{Lemma 3.4} we assumed that 
$A\in \mathcal{F}_\mu \cap \mathcal{F}_\nu$, in the foregoing proof we simply assumed $A\in \mathcal{F}_\nu$. Moreover, in Proposition~\ref{Lemma 3.4} $h_A\in\mathcal{H}$, and this is not the case in Theorem~\ref{Thm: Main}.
\end{remark}

\begin{proof}[(b)$\Rightarrow $(a)]
Let us assume that, on a measurable space $(\Omega ,\bm{\Sigma})$, any compatible pair~$(\mu,\nu)$ of measures on~$(\Omega ,\bm{\Sigma})$ satisfies also the Radon-Nikodym property.
We want to show that $(\Omega ,\bm{\Sigma},\mu)$ is weakly localizable.\\
Let us consider a family $\mathcal{E}$ ${\subseteq}$ $\mathcal{F}_\mu$.

\noindent 
We want to prove the existence of a $\mu$-supremum~$U\in\bm{\Sigma}$ of the family $\mathcal{E}$.\\
We denote by $\mathcal{U}_\mathcal{E}$ the family of all finite unions of elements of $\mathcal{E}$.\\
Then, for every $M \in \bm{\Sigma}$, we define 
$\nu(M)= \sup\left\{\mu(M\cap D) : D \in \mathcal{U}_\mathcal{E}\right\}$.\\
As each $M\cap D$ is contained in M, we have that $\nu(M)\le \mu (M)$.\\
If $M, N \in\bm{\Sigma}$ are disjoint,
then it easy to verify that $\nu(M \sqcup N) = \nu(M) + \nu(N)$.
If a sequence $(M_n)$ of measurable sets is such that
$M_n \subseteq M_{n+1}$, 
one can verify that~$\displaystyle \nu(\cup_n M_n) = \lim_n \nu(M_n)$.
So, $\nu$ is a measure on $(\Omega , \bm{\Sigma})$.
We want to prove that the pair $(\mu,\nu)$ is compatible; so, by hypothesis, the pair $(\mu,\nu)$ would satisfy the Radon-Nikodym property.\\
Let us consider $M \in \bm{\Sigma}$ such that 
$
0<
\nu(M)=\sup\left\{\mu(M\cap D) : D \in \mathcal{U}_\mathcal{E} \right\}
<+\infty
$.\\
Then, there exists $A \in \mathcal{U}_\mathcal{E}$ such that 
$0< \mu(M \cap A) <+\infty$.\\
That same set $A$ is such that 
$0 < \nu(M \cap A) = \mu(M \cap A) < +\infty$.\\
So, the pair $(\mu,\nu)$ satisfies the Radon-Nikodym property.\\
This implies that there exists a density 
$g=g_{\nu,\mu} : \Omega \rightarrow [0,+\infty]$ such that,\\ 
for every $F \in \mathcal{F}_\nu$ , $\nu(F) = \int_F g\ d\mu$.
Now, we want to prove that the set
$\{g = 1\} = U$ is a $\mu$-supremum of $\mathcal{E}$.\\
Let 
$E \in \mathcal{E}$; we define, for each $n \in \mathbb{N}$, 
$\displaystyle E_n = E \cap \left\{ g \le  1-\frac{1}{n}\right\}$.\\
As 
$\mu(E_n) = \nu(E_n) = \int_{E_n} g d\mu
\le 
\left(1-\frac{1}{n}\right) \mu(E_n)$,
it follows that $\mu(E_n)=0$.\\
Since $n$ is arbitrary, 
$E \cap \{ g < 1\} \in \mathcal{N}_\mu$.
Similarly, 
$E \cap \{ g > 1\} \in \mathcal{N}_\mu$.\\
Hence, 
$E \backslash \{g = 1\} 
=
E \cap \big(\{ g < 1\} \sqcup \{ g > 1\}\big) \in\mathcal{N}_\mu$.\\
Finally, let 
$F \in \bm{\Sigma}$ be such that for every 
$E \in \mathcal{E}$, $F \cap E \in \mathcal{N}_\mu$.\\ 
We find 
$0 = \nu(F \cap U) 
= 
\int_{F \cap U} g d\mu 
=
\int_{F \cap U} d\mu 
=
\mu(F \cap U)$.
So, $F\cap U \in \mathcal{N}_\mu$.
\end{proof}

\section*{Acknowledgments}
We would like to express our sincere gratitude to the anonymous referee for the very careful reading of our manuscript and for the constructive feedback. The referee's detailed comments and insightful suggestions have been invaluable in improving the clarity and rigor of this work.
\section*{ORCID}
\noindent Paolo Roselli - \url{https://orcid.org/0000-0003-3834-9358}


\begin{thebibliography}{0}
\bibitem{BrownPearcy1977} A. Brown, C. Pearcy, \emph{Introduction to Operator Theory I}, Springer New York, NY, 1977.
\bibitem{DePauw2024} T. De Pauw, ``Undecidably semilocalizable metric measure spaces", \emph{Commun. Contemp. Math.}, Vol.26, No.4, 2024.
\bibitem{Halmos1974} P. R. Halmos, \emph{Measure Theory}, Springer New York, NY, 1974.
\bibitem{McShane1962} E. J. McShane,
``Families of measures and representations of algebras of operators", \emph{Trans. Amer. Math. Soc.}, Vol.102, pp. 328--345, 1962.
\bibitem{OkadaRicker2019} S. Okada, W. J. Ricker, ``Classes of Localizable Measure Spaces'', in Buskes, G., et al. \emph{Positivity and Noncommutative Analysis}. Trends in Mathematics. Birkhäuser, Cham, 425--469, 2019.
\bibitem{Rao2004} M. M. Rao, \emph{Measure theory and integration},
Marcel Dekker, Inc., 2nd Ed., 2004.
\bibitem{Schep2003} A. R. Schep, ``And Still One More Proof of the Radon-Nikodym Theorem'', \emph{The Amer. Math. Monthly}, 110:6, 536--538, 2003.
\bibitem{Segal1951} I. E. Segal,
``Equivalence of measure spaces", \emph{Amer. J. Math.}, pp. 275--313, 1951.
\bibitem{Willem2022} M. Willem, \emph{Functional Analysis. Fundamentals and Applications}, Birkhäuser, 2nd Ed., 2022.
\end{thebibliography}
\end{document}